\documentclass[12pt]{amsart}
\usepackage{amsmath,amsthm,latexsym,amscd,amsbsy,amssymb,url}
\usepackage[all]{xypic}
 \relax

\chardef\bslash=`\\ % p. 424, TeXbook

\makeatletter
\def\verbatim{\interlinepenalty\@M \@verbatim
  \leftskip\@totalleftmargin\advance\leftskip2pc
  \frenchspacing\@vobeyspaces \@xverbatim}
\makeatother
\newtheorem{thm}{Theorem}[section]
\newtheorem{cor}[thm]{Corollary}
\newtheorem{lem}[thm]{Lemma}
\newtheorem{pro}[thm]{Proposition}

\newtheorem{ex}[thm]{Example}
\newtheorem{que}[thm]{Question}

\begin{document}

%%%%%%% Begin Topmatter %%%%%%%%%%

\title
{Three Questions on Special Homeomorphisms on Subgroups of $\mathbb R$ and $\mathbb R^\infty$}
\author{Raushan  Buzyakova}
\email{Raushan\_Buzyakova@yahoo.com}
\author{James West}
\email{west@math.cornell.edu}
\address{Department of Mathematics, Cornell University, Ithaca, NY, 14853}


%%%%%%% End topmatter %%%%%%%%%

\begin{abstract}{
We provide justifications for two questions on special maps on subgroups of $\mathbb R$. We will show that the questions can be treated from different points of view. We also discuss two versions of Anderson's Involution Conjecture.
}
\end{abstract}

\maketitle
\markboth{R. Buzyakova and J. West}{Three Questions on Special Homeomorphisms on Subgroups of $\mathbb R$ and $\mathbb R^\infty$}
{ }

\section{Questions}\label{S:questions}

\par\bigskip
In this note, the authors would like to share some of their observations that led to three questions that may be of interest to researchers from different areas of mathematics.  In what follows  by $\mathbb R$ we denote the topological additive group  of real numbers endowed with the Eucledian topology. Recall that a topological group $G$  is a topological space $X_G$ with a group operation $\cdot $  such that both $\cdot$ and the inversion $x\mapsto x^{-1}$ are continuous with respect to the topology of $X_G$. To avoid too many variables, we will refer to a topological group under consideration and its underlying topological space by the same letter. For example, if we a topological group $G$ is in our scope and we introduce another group operation $\oplus$ on the underlying space of $G$, we will  write $\langle G, \oplus\rangle$ instead of $\langle X_G, \oplus\rangle$.  Unless specified otherwise, the group operation on a group under consideration will be denoted by $+$. In \cite{B}, it is established that given a strictly  monotonic bijection $f$  on a $\sigma$-compact  subgroup $G$ of $\mathbb R$ it is possible to introduce another group structure on $G$ compatible with the topology of $G$ so that the new group is topologically isomorphic to $G$ and has $f$ as a shift. In fact, the proof in \cite{B} holds not only for $\sigma$-compact subgroups but also for subgroups with a certain strong homogeneity property to be discussed later. The mentioned result prompts the following question.

\par\bigskip\noindent
\begin{que}\label{que:shift}
Let $G$ be a subgroup of $\mathbb R$ and let $f:G\to G$ be a strictly monotonic homeomorphism. Is it true that there exists a binary operation  $\oplus$ on $G$ such that $G'=\langle G, \oplus\rangle$ is topologically isomorphic to $G$ and $f$ is a shift in $G'$?
\end{que}  

\par\bigskip
 It is not hard to see that the existence of a topologically conscious algebraic restructuring of $G$ that makes $f$ a shift is equivalent to $f$ being topologically conjugate (or topologically equivalent ) to a shift in $G$. Recall that homeomorphisms $h$ and $f$ on a space $X$  are {\it topologically conjugate, or topologically equivalent,} if there exists a homeomorphism $t$ on $X$ such that $tf=ht$. We, thus, can reformulate Question \ref{que:shift} as follows:

\par\bigskip\noindent
{\bf Re-formulation of Question \ref{que:shift}}
{\it Let $G$ be a subgroup of $\mathbb R$ and let $f:G\to G$ be a strictly monotonic homeomorphism. Is it true that $f$ is topologically conjugate to some shift in $G$?
}

\par\bigskip
 Even though a proof of equivalence of  Question  \ref{que:shift} and its version in the language of conjugate maps  is straightforward, we will  provide it in Section \ref{section:conjugate} for completeness.

\par\bigskip
As  shown in \cite{B}, the monotonicity requirement cannot be replaced by being free of periodic points. Since the example in  \cite{B} is unnecessary complicated, we use this opportunity to offer a short and transparent construction justifying the monotonicity requirement.

\par\bigskip\noindent
\begin{ex}
There exists a periodic-point free map on the integers $\mathbb Z$ that cannot be a shift in any group structure on $\mathbb Z$  isomorphic to $\mathbb Z$.
\end{ex}
\begin{proof} First note that any shift on $\mathbb Z$ has only finitely many orbits.
Let $\{Z_n:n=1,2,...\}$ be a partition of $\mathbb Z$, where each $Z_n$ is unbounded both from below and above. Define $f_n:Z_n\to Z_n$ by letting $f_n(x) =\min \{y\in Z_n: y>x\}$. Next  put $f=\cup f_n$. Clearly, $f$ is a periodic-point free homeomorphism.  Since  $f$ has infinitely many orbits $\{Z_n\}_n$, no group structure on $\mathbb Z$ isomorphic to $\mathbb Z$ has $f$ as a shift. In other words, $f$ is not topologically equivalent to any shift of $\mathbb Z$.
\end{proof}

\par\bigskip
We next shift our attention on the inversion $x\mapsto -x$ in a subgroup of $\mathbb R$. Recall that $f:X\to X$ is an involution if $f\circ f$ is the identity map on $X$. An argument similar to one in \cite{B} can be used to show that given a single fixed-point involution  $f$  on a $\sigma$-compact  subgroup $G$ of $\mathbb R$ it is possible to introduce another group structure on $G$ compatible with the topology of $G$ so that the new group is topologically isomorphic to $G$ and has $f$ as the inversion. For completeness purpose we will provide a proof  of this fact in Section \ref{section:inverse}. The argument fails to reach a desired conclusion for any subgroup of $\mathbb R$, and thus, the question is in order:

\par\bigskip\noindent
\begin{que}\label{que:inverse}
Let $G$ be a subgroup of $\mathbb R$ and let $f:G\to G$ be a single fixed-point involution. Is it true that there exists a binary operation  $\oplus_f$ on $G$ such that $G'=\langle G, \oplus_f, \rangle$ is topologically isomorphic to $G$ and $f$ is taking the additive inverse  in $G'$?
\end{que}

\par\bigskip\noindent
{\bf Re-formulation of Question \ref{que:inverse}}
{\it Let $G$ be a subgroup of $\mathbb R$ and let $f:G\to G$ be a single fixed-point involution. Is it true that $f$ is conjugate to taking the additive inverse in $G$?
}

\par\bigskip
The equivalence of Question \ref{que:inverse} and its re-formulation will be shown in in Section \ref{section:conjugate} too.
As mentioned earlier, the results leading to the above questions hold not only for $\sigma$-compact subgroups of $\mathbb R$ but also for all zero-dimensional subgroups with the property that every two non-empty open sets are homeomorphic. As shown in \cite{vD},  this property need not be present in a dense zero-dimensional subgroup.

\par\bigskip
For our third question, let $I=[-1,1]$, $I^{\infty}=\Pi_{i\ge 1}I_i$, and $\Bbb R^{\infty}=\Pi_{1\ge 1}R_i$,
For $X=I^{\infty}\text{ or }\Bbb R^{\infty}$, denote 
elements by $\bar x=(x_1,x_2,\dots)$, and let $\sigma (\bar x)=-\bar x$.  R.D. Anderson's 
Involution Conjecture for $\Bbb R^{\infty}$ or $I^{\infty}$ is that all involutions that have 
a single fixed point are topologically conjugate with $\sigma$. 

The initial paper on involutions of $I^{\infty}$ is \cite{Wo}, which proves that the Anderson 
 Conjecture is true for 
$I^{\infty}$ if the fixed point of every such involution has a basis of invariant, contractible neighborhoods.  In \cite{WWo}, the second author and Wong extended this by showing that the Anderson Conjecture is true for 
$I^{\infty}$ if the orbit space $I^{\infty}/\alpha$ of every involution on $I^{\infty}$ with a 
single fixed point is an absolute retract. 

In Proposition \ref{pro:anderson} we will show that the $\mathbb R^\infty$-version of the conjecture implies the $I^\infty$-version of the conjecture.
Since $\mathbb R^\infty$ is a topological group, a homeomorphism $f$ on $\mathbb R^\infty$ is topologically conjugate to $\sigma$ if and only if there exists a group operation $\oplus$ on $\mathbb R^\infty$ such that $\langle \mathbb R^\infty, \oplus\rangle$ is topologically isomorphic to $\mathbb R^\infty$ and has $f$ as the inversion (Section \ref{section:conjugate}). This alternative point of view on Anderson's Conjecture for $\mathbb R^\infty$  may give some extra routes for a positive  resolution.

\par\bigskip\noindent
\begin{que}
Is it true that the Anderson Conjecture for $I^\infty$ implies the Anderson Conjecture for $\mathbb R^\infty$?
\end{que}

\begin{pro}\label{pro:anderson}  The Anderson Conjecture for $\Bbb R^{\infty}$ implies the Anderson Conjecture for $I^{\infty}$.
\end{pro}
\begin{proof}
Let $\alpha$ be an involution of 
$I^{\infty}$ with a single fixed point.  Set $X=I^{\infty}\times \Bbb R^{\infty}$ and let
$\tau =\alpha \times \sigma:X\to X$ be the product (diagonal) involution. Define
$r:X\to I^{\infty}\times \{\bar 0\}$ by $r(\bar x,\bar y)=(\bar x,\bar 0)$.  Then $r$ is an 
equivariant retraction.  By \cite{M},  
%van Mill's book [M] probably has this.  My copy is at my office.
$X$ is 
homeomorphic to $\Bbb R^{\infty}$.  
It is not difficult to see using the linearity of 
$\sigma$ that $R^{\infty}/\sigma$ is an absolute retract for metric spaces. 
By the Anderson Conjecture for $\Bbb R^{\infty}$,
$\tau$ is topologically conjugate with $\sigma$, so $X/\tau$ is homeomorphic to 
$\Bbb R^{\infty}/\sigma$ and is therefore an absolute retract for metric spaces.  
Therefore the retraction $\bar r:X/\tau\to I^{\infty}/\alpha$ induced by $r$ shows that 
$I^{\infty}/\alpha$ is an absolute retract.  By \cite{WWo}, $\alpha$ is topologically conjugate to 
$\sigma$ and the Anderson Conjecture is true for $I^{\infty}$.
\end{proof}

\par\bigskip
In the next two sections we will provide the promised arguments that back up our claims.
We use standard notations and terminology. For topological basic facts and terminology one can consult \cite{Eng}. Since we do not use any intricate algebraic facts, any abstract algebra textbook is a sufficient reference.

\par\bigskip
\par\bigskip
\section{Maps Conjugate to Shifts or Inversions}\label{section:conjugate}

\par\bigskip
The goal of this section is to show that a homeomorphism $f$ on a group $G$ is topologically equivalent to the inversion of $G$ (or a shift)  if and only if there exists a group operation $\oplus$ on $G$ such that $G'=\langle G, \oplus\rangle$ is topologically isomorphic to $G$ and $f$ is the inversion (respectively,  a shift) in $G'$.  We will first achieve our goal for the case of inversions. All statements of this section are of  folklore nature but we present them with proofs for completeness.

In the arguments  of this section we will have to juggle several group structures on the same topological space. To keep track of where we invert, for a given group $G$, by $m_G$ we denote the operation of taking the inverse.
\par\bigskip\noindent
\begin{lem}\label{lem:twogroups}
Let $G_1=\langle G, \oplus_1, \rangle$ and $G_2=\langle G, \oplus_2, \rangle$ be topologically isomorphic
groups. Then $m_{G_1}$ and $m_{G_2}$ are topologically equivalent.
\end{lem}
\begin{proof} Let us denote by $(-_i)x$ the inverse of $x$ in $G_i$.
Let $h:G_2\to G_1$ be a group isomorphism, which is also a homeomorphism. Fix $x\in G$. We need to show that $m_{G_2}(x)=h^{-1}m_{G_1}h(x)$. 
\begin{description}
	\item[\it Left-hand side] $m_{G_2}(x)= (-_2)x$.
	\item[\it Right-hand side] Put $h(x)=y$. Then $m_{G_1}h(x) = (-_1)y$. Now put $h^{-1}((-_1)y)=z$. We have $h(x) = y$, $h(z)=(-_1)y$, and $h$ is an isomorphism. Therefore, $z=(-_2)x$.
\end{description}
The proof is complete.
\end{proof}

\par\bigskip\noindent
\begin{cor}\label{cor:equiv1}
Let $G$ be a topological group and $f:G\to G$  a homeomorphism. Then $f$ is topologically equivalent to the inversion of $G$, if there exists a group operation $\oplus$ such that $G'= \langle G, \oplus\rangle$ is topologically isomorphic to $G$ and $f$ is the inversion of $G'$.
\end{cor}

\par\bigskip\noindent
We now need to reverse the statement of Corollary \ref{cor:equiv1}. For this let us state the following  facts.
\par\bigskip\noindent
{\bf  Facts.} {\it
Let $G$ be a group, $X$ a set, and $f:X\to G$ a bijection. Then the following hold.
\begin{enumerate}
	\item $\langle X, \oplus_f\rangle$ is group, where $x \oplus_f y = z$ if and only of $f(x) + f(y) = f(z)$.
	\item $f:\langle X, \oplus_f\rangle \to G$ is an isomorphism.
\end{enumerate}

}
\par\bigskip\noindent
\begin{lem}\label{lem:equiv2}
Let $G$ be a topological group and let $f:G\to G$ be a homeomorphism topologically equivalent to $m_G$. Then there exists a binary operation $\oplus$  on $G$ such that $G'=\langle G, \oplus\rangle$ is a topological group topologically isomorphic to $G$ and $f=m_{G'}$.
\end{lem}
\begin{proof}
Fix a homeomorphism $h:G\to G$ such that $f=h^{-1}m_Gh$. Put $\oplus = \oplus_h$, where $\oplus_h$ is as in Fact (1).
Since $h$ is a homeomorphism, by Fact (2), the groups $G$ and $G'$ are topologically isomorphic by virtue of $h$. It remains to show that $f= m_{G'}$.  Fix $x$ and put $y=h(x)$. Then $f(x) = h^{-1}m_Gh(x)=h^{-1}m_G(y)=h^{-1}(-y)$. Since $h$ is an isomorphism and $y=h(x)$, we have $h^{-1}(-y)$ is the inverse of $x$ with respect to $\oplus$, which is $m_{G'}(x)$.
Lemma is proved.
\end{proof}

\par\bigskip\noindent
We now summarize Corollary \ref{cor:equiv1} and Lemma \ref{lem:equiv2} as follows.

\par\bigskip\noindent
\begin{thm}\label{thm:equivdef}
A homeomorphism  $f$ on a group $G$ is topologically equivalent to the inversion of $G$ if and only if there exists a group operation $\oplus$ on $G$ such that $G'=\langle G, \oplus\rangle$ is topologically isomorphic to $G$ and $f$ is the inversion of $G'$.
\end{thm}

\par\bigskip\noindent
\par\bigskip\noindent
We will next show that a version of Theorem \ref{thm:equivdef} for shifts holds too.

\par\bigskip\noindent
\begin{lem}\label{lem:nonneutralshiftsareconjugate}
Let $G_1=\langle G, \oplus_1\rangle$ and $G_2=\langle G, \oplus_2\rangle$ be topologically isomorphic
groups. Let $f_1$ be a non-neutral shift in $G_1$. Then $f_1$ is topologically equivalent to some  non-neutral shift $f_2$ in $G_2$.
\end{lem}
\begin{proof} Let $c_1$ be such that $f_1(x) = x\oplus_1 c_1$ for all $x$ in $G_1$.  By hypothesis, there exists a topological isomorphism $h: G_2\to G_1$. Since $c_1$ is not the neutral element of $G_1$, we conclude that $c_2=h^{-1}(c_1)$ is not the neutral element of $G_2$. Define $f_2$ by letting $f_2(x) = x\oplus_2 c_2$ for all $x$ in $G_2$.
It suffices to show that $f_2(x)=h^{-1}f_1h(x)$.
\begin{description}
	\item[\it Left-hand side] $f_2(x)= x\oplus_2c_2$.
	\item[\it Right-hand side] Put $h(x)=y$. Then $f_1h(x) = y\oplus_1 c_1$. Since $h^{-1}$ is an isomorphism, we obtain
$h^{-1}(y\oplus_1c_1) = h^{-1}(y) \oplus_2 h^{-1}(c_1) = x \oplus_2 c_2$.
\end{description}
The proof is complete.
\end{proof}

\par\bigskip\noindent
\begin{cor}\label{cor:shiftisconjugatetoshift}
Let $G$ be a topological group and let $f:G\to G$ be a homeomorphism. Then $f$ is topologically equivalent to a shift in $G$ if there exists a group operation $\oplus$ on $G$ such that $G'=\langle G, \oplus\rangle$ is topologically isomorphic to $G$ and $f$ is a shift in $G'$.
\end{cor}

\par\bigskip\noindent
\begin{lem}\label{lem:conjugatetoshiftisshift}
Let $G$ be a topological group and let $f:G\to G$ be a homeomorphism topologically equivalent to a non-neutral shift in $G$. Then there exists a group operation $\oplus$  on $G$ such that $G'=\langle G, \oplus\rangle$ is a topological group topologically isomorphic to $G$ and $f$ is a non-neutral shift in $G'$.
\end{lem}
\begin{proof}
Fix a homeomorphism $h:G\to G$ and a non-neutral shift $g$ determined by a non-neutral constant $c$ such that $f=h^{-1}gh$. Put $\oplus = \oplus_h$, where $\oplus_h$ is as in Fact (1).
Since $h$ is a homeomorphism, by Fact (2), the groups $G$ and $G'=\langle G, \oplus, \mathcal T_G\rangle$ are topologically isomorphic by virtue of $h: G'\to G$. It remains to show that $f$ is the shift determined by $d=h^{-1}(c)$.  Fix $x$ and put $y=h(x)$. Then 
$f(x) = h^{-1}gh(x)=h^{-1}g(y)=h^{-1}(y+c)$. Since $h$ is an isomorphism and $y=h(x)$, we have $h^{-1}(y+c)$ is equal to $h^{-1}(y)\oplus h^{-1}(c)=x\oplus d$.
Lemma is proved.
\end{proof}

\par\bigskip\noindent
We now summarize Corollary \ref{cor:shiftisconjugatetoshift} and Lemma \ref{lem:conjugatetoshiftisshift} as follows.

\par\bigskip\noindent
\begin{thm}\label{thm:equiv4shifts}
A homeomorphism  $f$ on a group $G$ is topologically equivalent to a non-neutral shift if and only if there exists a group  operation $\oplus$ on $G$ such that $G'=\langle G, \oplus\rangle$ is topologically isomorphic to $G$ and $f$ is a non-neutral shift in $G'$.
\end{thm}

\par\bigskip
\section{Single Fixed-Point Involutions on Subgroups of $\mathbb R$}\label{section:inverse}

\par\bigskip
In this section we modify the argument in \cite{B} to show that given a single fixed-point involution $f$ on a $\sigma$-compact subgroup $G$ of $\mathbb R$, one can restructure the algebra of $G$ without changing the topology so that the resulting group is topologically isomorphic to $G$ and has $f$ as the inversion. For convenience, we copy some  useful facts from Section \ref{section:conjugate}:
\par\bigskip\noindent
{\bf Facts.} Let $G$ be a group, $X$ a set, and $f:X\to G$ a bijection. Then the following hold.
\begin{enumerate}
	\item $\langle X, \oplus_f\rangle$ is a group, where $x \oplus_f y=z$ if and only if $f(x) + f(y) = f(z)$.
	\item $f: \langle X, \oplus_f\rangle\to G$ is an isomorphism.
	\item If $G$ is a topological group, $X$ is a topological space, and $f$ is a homeomorphism, then 
$f:\langle X, \oplus_f\rangle\to G$ is a topological isomorphism.
\end{enumerate}

\par\bigskip
\begin{thm}\label{thm:main}
Let $G$ be a $\sigma$-compact subgroup of  $\mathbb R$ and let $f:G\to G$ be a continuous involution with exactly one fixed point. Then there exists a group operation $\oplus$ on $G$ such that $G'=\langle G,\oplus\rangle$ is topologically isomorphic to $G$ and $f$ is the inversion in $G'$.
\end{thm}
\begin{proof}
Put $A= \{a\in G: a> f(a)\}$. Let $e$ be the fixed point of $f$. 

\par\bigskip\noindent
{\it Claim 1. $A$ is open.}
\par\bigskip\noindent
{\it Claim 2.  $f(A)\cap A=\emptyset$.}
\par\smallskip\noindent
To prove the claim recall that $f$ is an involution. Therefore, $x < f(x)$ for each $x\in f(A)$. The claim is proved.
\par\bigskip\noindent
{\it Claim 3.  $G=A\cup f(A)\cup \{e\}$.}
\par\smallskip\noindent
The statement of the claim follows from the fact that $f(x)<x$, or $f(x)>x$, or $f(x)=x$. In the first case, $x\in A$. In the second case $x\in f(A)$. Since $e$ is the only fixed point of $f$, the third case implies that $x=e$. The claim is proved.
\par\bigskip\noindent
{\it Claim 4. There exists a homeomorphism $h: A\cup \{e\}\to G\cap [0,\infty )$ such that $f(e) = 0$.}
\par\smallskip\noindent
To prove the claim we first assume that $e$ is not a limit point of $G$. Then, $G$ is a closed discrete countable space and the conclusion is obvious. Assume now that $G=\mathbb R$. Since $e$ is the only fixed point,   $A=(e,\infty)$ and the conclusion is obvious. If $G$ is neither discrete nor connected, then  $G$ is a zero-dimensional, $\sigma$-discrete, dense in itself subset of $\mathbb R$. Therefore, $G$ is homeomorphic to the set of rational numbers or to the product of the Cantor Set and the set of irrational numbers (\cite{AU} and \cite{Sie}, or \cite{vM0}). Therefore, any two non-empty subsets of $G$ are homeomorphic. Hence, there exists a required homeomorphism. The claim is proved.

\par\bigskip\noindent
Fix $h$ is in Claim 4 and define $\tilde h: G\to G$ as follows:
$$
\tilde h(x) = \left\{
        \begin{array}{ll}
             h(x) & x \in A\cup \{e\} \\
           -h(f(x))& x\in f(A)
        \end{array}
    \right.
$$

\par\bigskip\noindent
{\it Claim 5. $\tilde h$ is a homeomorphism.} 
\par\smallskip\noindent
By virtue of Claims 1-3, it suffices to show that $h(e)=-h(f(e))$. By the choice of $h$, we have $h(e)=0$. Since $f$ fixes $e$, we have $-h(f(e))=-h(e)=0$. The claim is proved.

\par\bigskip\noindent
Put $\oplus = \oplus_{\tilde h}$. Then by Fact (3), $G'=\langle G, \oplus, {\mathcal T}_G\rangle$ is topologically isomorphic to $G$ by virtue of $\tilde h$.

\par\bigskip\noindent
{\it Claim 6. $f$ is the operation of taking the  inverse in $G'$.}  Pick $x$ in $G'$. We need to show that $f(x)\oplus  x=x\oplus f(x) =e$. Since $G$ is commutative, it suffices to demonstrate only the rightmost equality. By Claim 3, $x$ is in $A$, or in $f(A)$ or is equal to $e$. Let us consider each case separately.
\begin{description}
	\item[\rm Case $x\in A$] We have $x\oplus f(x) = x \oplus_ {\tilde h}f(x) = z$, where 
$\tilde h(x) + \tilde h(f(x)) =\tilde h(z)$. Since $f(x)\in f(A)$, we have $\tilde  h(f(x))=-\tilde h((f(f(x)))$. Since $f$ is an involution,  we have $\tilde h(f(x))=-\tilde h(x)$. Therefore, $\tilde h(x) + \tilde h(f(x)) = \tilde h(x) - \tilde h(x) = 0=\tilde h(e)$. Therefore, $z=e$ and  $x\oplus f(x) = e$.
	\item[\rm Case $x\in f(A)$] We have $x\oplus f(x) = x \oplus_ {\tilde h}f(x) = z$, where 
$\tilde h(x) + \tilde h(f(x)) =\tilde h(z)$. Since $x\in f(A)$, we have $\tilde  h(x)=-h(f(x))$. Since $f(x)\in A$, we have $\tilde h(f(x))= h(f(x))$. Therefore, $\tilde h(z) = 0$, meaning $z=e$.
	\item[\rm Case $x=e$] Since $h(e) = 0$. We have $\tilde h (e) + \tilde h(e) = 0 +0 = \tilde h(e)$. Hence $e\oplus e = e$.
\end{description}
The proof is complete.
\end{proof}

\par\bigskip\noindent
\par\bigskip\noindent
{\bf Acknowledgment.} The authors would like to thank the referee for valuable remarks and corrections.

\par

\end{document}